\documentclass[12pt]{amsart}
\usepackage[utf8]{inputenc}

\usepackage{geometry} 
\usepackage{amsmath}
\usepackage{amssymb}
\usepackage{amsthm}
\usepackage{xcolor}
\usepackage{cite}
\usepackage{enumerate, enumitem}
\usepackage{soul}
\setlist{nolistsep}

\numberwithin{equation}{section}
\numberwithin{figure}{section}
\numberwithin{table}{section}

\theoremstyle{definition}
\newtheorem{thm}{Theorem}[section]

\newtheorem{lem}[thm]{Lemma}
\newtheorem{prop}[thm]{Proposition}

\newtheorem{rem}[thm]{Remark}

\newcommand{\thmref}[1]{Theorem \ref{#1}}

\newcommand*{\hilbsp}{\mathcal{H}}

\newcommand*{\boundop}{B(\hilbsp)}

\newcommand*{\nest}{\mathcal{L}}

\newcommand*{\nestalg}{\mathcal{T(\nest)}}

\newcommand*{\nats}{\mathbb{N}}

\newcommand*{\tens}[2]{#1 \otimes #2}
\newcommand*{\module}{\mathcal{M}}

\newcommand*{\inner}[2]{\langle #1,#2 \rangle}

\newcommand*{\union}{\bigcup}

\newcommand*{\norm}[1]{||#1||}

\newcommand*{\ortho}[1]{#1^\perp}

\newcommand*{\jproj}[1]{\hat{#1}}
\newcommand*{\lieprod}[2]{[#1,#2]}

\newcommand*{\llinner}[1]{\Omega_T}
\newcommand*{\xinn}{\psi_T}

\newcommand*{\yinn}{\varphi_T}

\newcommand{\F}{\mbox{${\mathcal F}$}}

\renewcommand{\H}{\mbox{${\mathcal H}$}}

\newcommand{\M}{\mbox{${\mathcal M}$}}
\newcommand{\N}{\mbox{${\mathcal L}$}}

\renewcommand{\S}{\mbox{${\mathcal S}$}}
\newcommand{\T}{\mbox{${\mathcal T}$}}

\newcommand{\X}{\mbox{${\mathcal X}$}}

\newcommand{\Bil}{\operatorname{BIL}}

\DeclareMathOperator{\image}{\text{im}}

\newcommand*{\lt}{\mathcal{L}}

\title[Kernel maps and operator decomposition]{Kernel maps and operator decomposition}

\author[G. Matos]{Gabriel Matos}
\address{Instituto Superior T\'ecnico, Universidade de Lisboa\\
Av. Rovisco Pais, 1049-001 Lisboa, Portugal}
\email{gabriel.matos@tecnico.ulisboa.pt}

\author[L. Oliveira]{Lina Oliveira}
\address{Center for Mathematical Analysis\\
Geometry and Dynamical Systems\\
{\sl{and}}
Department of Mathematics\\
Instituto Superior T\'ecnico \\ 
Universidade de Lisboa\\
Av. Rovisco Pais\\
1049-001 Lisboa, Portugal}
\email{linaoliv@math.tecnico.ulisboa.pt}

\thanks{The first author was supported by a Calouste Gulbenkian Foundation grant through the \emph{Novos Talentos em Matem\'atica} program. The second author was partially supported by FCT/Portugal grants UID/MAT/04459/2013 and EXCL/MAT-GEO/0222/2012} 


\begin{document}
\begin{abstract}We introduce the notions of kernel map and  kernel set of a bounded linear operator on a Hilbert space relative to a subspace lattice. The characterization of the kernel maps and kernel sets of finite rank operators 
 leads to showing that every norm closed Lie module of a continuous nest algebra is decomposable.
The continuity of the nest cannot be lifted, in general.
\end{abstract}

\maketitle

\section{Introduction}
The main concepts introduced and investigated  in the present work are the kernel map and the kernel set of an operator  relative to a subspace lattice (cf. Section  \ref{s_kernelmap}).  The results obtained are  subsequently applied to extend \cite{CL, O2} by showing  that every  norm closed Lie module of a continuous nest algebra is decomposable. 

The idea of linking sets of operators with lattices and bilattices has been present in the literature, e.g., \cite{Deg, Erdos, Had, Halmos, LS, ST} and, more recently,  \cite{BO}. However, this has been done mainly in connection with reflexivity which is not the approach here, since we are only interested in a single operator at a time.

 Given a subspace lattice $\nest$ on a Hilbert space $\H$, we show  that each bounded linear operator on $\H$ determines a {\sl kernel map}, from $\nest$ to  a bilattice contained in $\nest\times \nest^\perp$, and a corresponding {\sl kernel set}. The properties of the kernel map and kernel set are crucial in addressing the question of  decomposability central to  Section \ref{s_decomp}.

A set $\S$ of bounded linear operators on a Hilbert space is called {\sl decomposable} if each finite rank operator in $\S$ is the sum of finitely many rank-1 operators in $\S$.   
 Decomposability has been investigated in a wide variety of settings and it is well-known that, for example, nest algebras are decomposable \cite{ringrose1965} whilst this is not the case for CSL algebras \cite{HM, Long, LL}. It is also the case that norm closed modules of nest algebras are decomposable \cite{erdosdensity, EP} but  that  if, however, one considers algebraic structures other the associative ones, this can fail even for ideals. Indeed, Lie ideals of nest algebras might not  be decomposable unless the nest is continuous  \cite{O2} or has, at most,   one   atom which must be infinite dimensional   \cite{CL}. In this latter case \cite{CL}, the Hilbert space is supposed to be separable whilst the setting of the former \cite{O2} is general.

 In Section 3,  we apply the results and techniques developed in Section 2 to investigate the decomposability of norm closed Lie modules of nest algebras and are bound, of course, by the same limitations as in the case of Lie ideals. We show that every norm closed Lie module over   a continuous nest algebra is decomposable (cf. Theorem \ref{rankndecomp}) and, although the result itself might be seen as an extension of \cite{CL, O2}, it is in fact obtained using  an approach  quite distinct from those of \cite{CL, O2}. The hypothesis of the continuity of the nest cannot be dispensed with, in general.

We end this section establishing some notation and facts needed in the two subsequent sections.\\

Recall that the set ${\mathcal P}(\mathcal{H})$ of self-adjoint projections on a Hilbert space is naturally endowed with a partial order relation: given self-adjoint projections $P, Q$, define $P\leq Q$, if $PQ=P$. The set ${\mathcal P}(\mathcal{H})$ together with this  partial ordering is a \emph{lattice}, i.e., it is closed for binary infima and suprema.  Given $P, Q\in {\mathcal P}(\mathcal{H})$, the supremum $P\vee Q$ and the infimum $P\wedge Q$ are, respectively, the projection onto the closure of the span of the ranges of $P$ and $Q$ and the projection onto the intersection of those ranges.

A lattice $\nest$ of self-adjoint projections on $\H$ is said to be a \emph{subspace lattice} if it contains $0$, $I$ and is strongly closed.  
 Letting $P^\perp=I-P$, we define the subspace lattice $\nest^\perp$ by 
 $
 \nest^\perp=\{P^\perp\colon P\in \nest\}.
 $
  Subspace lattices are complete lattices.

A \emph{nest} is a totally ordered subspace lattice. The \emph{nest 
algebra} $\T(\N)$ associated 
with a nest $\N  $ is the subalgebra of all operators $T $ in $B  (\H )     $ such that, for all projections $P $ 
in $\N  $, 
 $T(P  (\H ))
\subseteq P  (\H ).
$
 For any   projection  $P$ in the nest $\N$, define the projection $P_{-}$
 (respectively, $P_{+}$) in $\N$  by  $P_{-} =\vee \{Q\in \N: Q<P\}$ 
(respectively, $P_{+} =\wedge \{Q\in \N: P<Q\}$). 
A \emph{continuous nest} is a nest where, for all $P\in\nest$, $P_-=P$. A nest algebra $\nestalg$ associated with a continuous nest is called a \emph{continuous nest algebra}. For more details on nest algebras, see \cite{davidsonbook,ringrose1965}.

Let $x$ and $y$ be elements of the Hilbert space $\hilbsp$ and  let $\tens{x}{y}$ be the rank one operator defined, for all $z$ in $\hilbsp$, by $z \to \inner{z}{y}x$, where $\inner{\cdot}{\cdot}$ denotes the inner product of $\hilbsp$. A rank one operator $\tens{x}{y}$ lies in $\nestalg$ if, and only if, there exists a projection $P$ such that $P_- y=0$ and $Px=x$. It follows that in a continuous nest algebra, $x$ and $y$ are mutually orthogonal.  It is shown in \cite{ringrose1965} that $P$ can be chosen to be equal to  $\wedge \{Q\in \nest: Qx=x\}$. 

In the sequel, Hilbert spaces are denoted by $\H$ and are all assumed to be complex and separable. By a slight abuse of notation, we shall not discriminate between a self-adjoint projection and its range. 

Let $\nest$ be a subspace lattice on $\H$. Generalizing \cite[Lemma 3.2]{O1},  for $x\in \H$, define  the projections  $P_{x,{\scriptsize \N}}$ and $\jproj{P}_{x,{\scriptsize \N}}$ by 
\begin{equation}\label{e33}
P_{x,{\scriptsize \N}} = \wedge\{P\in\nest:Px = x\},
\end{equation}  
\begin{equation}\label{e66}
\jproj{P}_{x,{\scriptsize \N}} = \vee\{P\in\nest:Px = 0\},
\end{equation} 
respectively. It is easily seen  that $\jproj{P}_{x,{\scriptsize \N}}x=0$ and $P_{x,{\scriptsize \N}}x=x$.
 In what follows, we shall omit the subscript ${\cdot}_{{\scriptsize \N}}$ whenever there is no risk of misunderstanding.

\begin{prop}\label{p_prelim} Let $\N$ be a nest in a Hilbert space $\H$ and let $S$ be a subset of  $\hilbsp$ not containing zero. The  following hold.
\begin{enumerate}[label=(\roman*)]
\item  Given   projections $N,P$ in $\N$ with $N<P$, there exist $x\in \H$ such that $P_x=P$ and $\hat{P}_x=N$.
\item  If, for all distinct $x,y \in S$, $P_x \neq P_y$, then $S$ is  a linearly independent set.
\end{enumerate}
\end{prop}
\begin{proof} 

(i) We begin by showing that, given $P\in \N$, $P=P_x$, for some $x\in \H$. If $P=0$, then $x=0$ satisfies the requirement. Consider now a non-zero projection $P\in \N$. If   $P_-< P$, then, for any non-zero $x\in P\ominus P_-$,   $P_x=P$. 

 If, on the other hand,  $P=P_-= \vee\{Q\in \N\colon Q<P\}$, then, since $\H$ is separable, by \cite[Proposition 3 and Corollary 4]{BCW},  there exists a strictly  increasing sequence $(Q_n)$ in $\N$ such that $P=\vee Q_n$. Choose a sequence $(x_n)$ such that  $x_1\in Q_1$,
 \begin{equation}\label{noteq}
  x_n\in Q_n\ominus Q_{n-1}\quad \mbox{and} \quad  1=\|x_1\|=\|x_n\|.
 \end{equation} 
  Let $x\in \H$ be  the sum of the  series 
 $
 \sum_{n=1}^\infty \frac{1}{n^2} x_n
 $.
   By the definition of  $x$, we have  $P_x\leq P$. But $P_x$ cannot be smaller than $P$. If this were the case, then there would exist $p\in \mathbb{N}$   such that $P_x< Q_{p+1}\leq P$. Hence, for all $n>p+1$,  $x_n=0$ thus contradicting \eqref{noteq}.

By definition, for any $w\in \H$, 
$$
\hat{P}_w 
= \vee\{P\in\nest:Pw = 0\}
=(\wedge \{P^\perp \in\nest^\perp:Pw = 0\})^\perp
=(\wedge \{P \in\nest^\perp: P w = w\})^\perp.
$$
Hence $\hat{P}_w =(P_{w,{\scriptsize \N^\perp}})^\perp$, where 
 $P_{w,{\scriptsize \N^\perp}}$ is  as in \eqref{e33} when considering the nest $\N^\perp$. Since, by the above, there exists $x\in \H$ such that $P_{x,{\scriptsize \N^\perp}}= P^\perp$, it follows that $\hat{P}_x=P$. 

It remains to show that given $N,P\in \N$ with $N<P$, there exists $x\in \H$ such that $P=P_x$ and $N=\hat{P}_x$. By the above, there exist $z,w\in \H$ such that $P_z=P$ and $\hat{P}_w=N$.  Let $x_1=(P-N)z$ and  $x_2=(\hat{P}_{x_1}-N)w$. Observe that     
$0\leq \hat{P}_{x_1}-N$. Then $x=x_1 + x_2$ is such that $P=P_x$ and $N=\hat{P}_x$.

(ii)    
To prove  that $S$ is linearly independent, we must show that, for every finite subset 
$M=\{x_i\colon i=1,\dots, m\}$ of  $S$, we have that $\sum_{i=1}^m \alpha_i x_i=0$ only if  $\alpha_i = 0$, for all $i=1,\dots,m$.
 The assertion trivially holds when $S$ is a singleton. We shall prove the general result using mathematical induction.

Assume now that the hypothesis is valid for all subsets of $S$ having $n$ or less vectors. Let $M_{n+1}$ be any subset of $S$ containing $n+1$ distinct vectors. Index the vectors such that $P_{x_i} < P_{x_j}$, $i < j$. Hence, if $\sum_{i=1}^{n+1} \alpha_i x_i = 0$, then 
$$
\sum_{i=1}^{n+1} \alpha_i \ortho{P_{n}} x_i =  \alpha_{n+1} \ortho{P_{n}} x_{n+1} = 0,
$$
from which follows that $\alpha_{n+1} = 0$.
 The assertion now follows by application of the induction hypothesis.
\end{proof}

\section{Kernel maps and kernel sets}\label{s_kernelmap}

Let ${\mathcal P}(\mathcal{H})$ be the lattice of projections in $B(\H)$.  
 Following \cite{BO, ST}, we consider the partial order  $\preceq$  on  ${\mathcal P}(\mathcal{H})\times {\mathcal P}(\mathcal{H})$   defined, for all
$(P_1,Q_1),$ $(P_2,Q_2)$ $\in {\mathcal P}(\mathcal{H})\times {\mathcal P}(\mathcal{H})$, by
\begin{equation} \label{eq04}
(P_1,Q_1)\preceq (P_2,Q_2)\qquad \text{if, and only if,}\qquad P_1\leq P_2\; \text{ and }\; Q_2\leq Q_1.
\end{equation}
It follows that the  operations of join  and meet  are given, respectively, by
\begin{equation} \label{eq03}
\begin{split}
(P_1,Q_1)\vee (P_2,Q_2)&=(P_1\vee P_2,Q_1\wedge Q_2),\\
(P_1,Q_1)\wedge (P_2,Q_2)&=(P_1\wedge P_2,Q_1\vee Q_2).
\end{split}
\end{equation}
We write ${\mathcal P}(\mathcal{H})\times_{\preceq} {\mathcal P}(\mathcal{H})$ when referring to   the cartesian product together with the partial order relation $\preceq$.  Recall that a subset  of ${\mathcal P}(\mathcal{H})\times_{\preceq} {\mathcal P}(\mathcal{H})$ is said to be a \emph{bilattice} if 
it is closed under the lattice operations \eqref{eq03} and contains the pairs $(0,0)$, $(0,I)$, and $(I,0)$ (cf. \cite{BO, ST}). The top and bottom elements of any bilattice are $(I,0)$, and $(0,I)$, respectively. 
${\mathcal P}(\mathcal{H})\times_{\preceq} {\mathcal P}(\mathcal{H})$ is itself one amongst many examples of bilattices.  Here, however, we are mainly interested in  bilattices  associated with a single operator in a way to be made precise below (see \eqref{01}). For more details on bilattices, the reader is referred to \cite{BO, Erdos, LS, ST}.

Let $\N$ be a  subspace lattice  and let $T$ be an operator in $B(\H)$. Adopting the notation of  \cite{BO}, we define the set $\Bil(T, \N)$ by 
\begin{equation}\label{01}
\Bil(T, \N)=\{ (P,Q)\in \mathcal{L}\times_{\preceq} \N^\perp\colon QTP=0\}.
\end{equation}  
Observe that the set defined in \eqref{01} is a strongly closed bilattice which coincides with the intersection of  the strongly closed bilattices $\mathcal{L}\times_{\preceq} \N^\perp$  and 
 $
 \{ (P,Q)\in {\mathcal P}(\mathcal{H})\times_{\preceq} {\mathcal P}(\mathcal{H})\colon QTP=0\}
 $
(cf. \cite{BO, ST}).   For each $P\in \N$, define  the projections $\varphi_{T, {\tiny \N}}(P), \psi_{T, {\tiny \N}}(P)$ 
by
\begin{equation} \label{eq1}
\varphi_{T, {\tiny \N}}(P)=\vee\{P' \in \lt \colon ( P', P^\perp)\in \Bil(T, \N)\},
\end{equation}
\begin{equation}\label{eq2}
\psi_{T, {\tiny \N}}(P)=\vee \{P'^\perp \in \lt^\perp \colon \varphi_{T, {\tiny \N}}(P')=\varphi_{T, {\tiny \N}}(P)\}, 
\end{equation}

\begin{lem}\label{rem1}
Let $\N$ be a  subspace lattice  and let $T$ be an operator in $B(\H)$. Then the map $\varphi_{T, {\tiny \N}}$  is an order homomorphism on $\N$ and the map $\psi_{T, {\tiny \N}}$ is an anti-order homomorphism from $\N$ to $\N^\perp$.
\end{lem}

\begin{proof} Since $\lt$ is complete, for each $P$,  $\varphi_{T, {\tiny \N}}(P)\in \nest$ and $\psi_{T, {\tiny \N}}(P)\in \nest^\perp$.
We show firstly that the map defined in \eqref{eq1} is an order homomorphism on the lattice $\N$. In fact, if $P_1,P_2, P'$ are projections in $\N$ such that $P_1\leq P_2$ and $(P',P_1^\perp)\in \Bil(T, \N)$, then 
$$
 P_2^\perp TP'=  P_2^\perp P_1^\perp TP'=0.
$$
Hence $(P',P_2^\perp)\in \Bil(T, \N)$, from which follows that $\varphi_{T, {\tiny \N}}(P_1)\leq \varphi_{T, {\tiny \N}}(P_2)$.

In order to prove that $\psi_{T, {\tiny \N}}$ is an anti-order homomorphism from $\N$ to $\N^\perp$, we will prove that
\begin{equation}\label{extra1}
  \varphi_T(\psi_{T}(P)^\perp)=\varphi_T(P),
  \end{equation}
  \begin{equation}\label{extra2}
  \psi_{T}(P) =\vee\{Q^\perp\in \N^\perp\colon Q^\perp T\varphi_{T}(P)=0\}.
\end{equation}
Observe that, by  \eqref{eq2}, $P^\perp \leq \xinn(P)$. Hence, since the map $\varphi_T$  preserves order,
$\yinn(\xinn(P)^\perp) \leq \yinn (P)$.  But, by Proposition \ref{l_bil}, 
$\xinn(P) T\yinn(P)  = 0$, which shows that $\yinn(\xinn(P)^\perp) \nless \yinn (P)$. Hence 
$$
\varphi_T(\psi_{T}(P)^\perp)=\varphi_T(P),
$$
proving \eqref{extra1}.

Since $\lt^\perp$ is complete, for each $P$,   we have that $\psi_{T}(P)\in \nest^\perp$. Define the set $\S$ by
$$
\S=\{Q^\perp \in \N^\perp\colon Q^\perp T\varphi_{T}(P)=0\}.
$$
By Proposition \ref{l_bil},  $(\varphi_{T} (P), \psi_{T} (P))\in \Bil(T, \N)$ from which follows that $\psi_{T}(P)\in\S$. Consequently, 
$
\psi_{T}(P)\leq \vee \S.
$
 Moreover, since $\N$ is strongly closed, 
 $$
 (\vee \S) T\varphi_{T} (P)=0,
 $$
 which shows that $\varphi_T(P)\leq \varphi_T((\vee \S)^\perp)$.
 
 By \eqref{eq3}, we also have 
 $$
 (\vee \S) T\varphi_{T} ((\vee \S)^\perp)=0.
 $$
 Hence, since 
 $
\psi_{T}(P)\leq \vee \S
$, 
$$
 \psi_{T}(P)T\varphi_{T} ((\vee \S)^\perp)=\psi_{T}(P)(\vee \S) T\varphi_{T} ((\vee \S)^\perp)=0.
 $$
 Taking into account  definition \eqref{eq1}, it follows from the equality above  and \eqref{extra1}  that 
 $$ 
 \varphi_{T} ((\vee \S)^\perp)\leq \varphi_T(\psi_{T}(P)^\perp)= \varphi_T((P)).
 $$
 Consequently, $\varphi_{T} ((\vee \S)^\perp)=\varphi_T(P)$ which shows that  
 $\vee \S$ lies in the defining set of \eqref{eq2}. 
 
 Since, by definition \eqref{eq2}, $\vee\S\leq  \psi_{T}(P)$, it follows that
 $$
\psi_{T} (P)=\vee\S=\vee\{Q^\perp\in \N^\perp\colon Q^\perp T\varphi_{T}(P)=0\},
$$
as required.

Now let $P_1, P_2\in \N$ be such that $P_1\leq P_2$. Since, for $Q$ such that $Q^\perp T\varphi_{T}(P_2)=0$, 
$$
Q^\perp T\varphi_{T}(P_1)=Q^\perp T\varphi_{T}(P_1)\varphi_{T}(P_2)=Q^\perp T\varphi_{T}(P_2)\varphi_{T}(P_1)=0,
$$
 we have 
$$
\{Q^\perp\in \N^\perp\colon Q^\perp T\varphi_{T}(P_2)=0\}\subseteq \{Q^\perp\in \N^\perp\colon Q^\perp T\varphi_{T}(P_1)=0\}.
$$
Hence, by \eqref{extra2}, 
$
\psi_{T} (P_2)\leq \psi_{T} (P_1),
$
which shows that the map $\psi_{T}\colon \N\to \N^\perp$ is an anti-order homomorphism.
\end{proof}

In what follows,  we  mostly omit the subscript ${\cdot}_{\tiny \N}$ to simplify the notation. 

\begin{prop}\label{l_bil} Let $T$ be an operator in $B(\H)$, let  $\N$ be a  subspace lattice and let  $P$ be a projection in $\N$. Then 
$$(\yinn(P), \xinn(P))\in \Bil(T, \N),$$  
\begin{equation}\label{e2}
(\yinn(P), P^\perp)=\vee\{(P',Q)\in \Bil(T, \N)\colon P^\perp\leq Q\},
\end{equation}
and
\begin{equation} \label{e1}
(\yinn(P),\xinn(P))=\wedge\{(\yinn(P'), P'^\perp)\in \Bil(T, \N)\colon \yinn(P')=\yinn(P)\}.
\end{equation}
\end{prop}

\begin{proof} Since $\lt$ is complete, $(\yinn(P),\xinn(P))\in \N\times\N^\perp$. To show that $(\yinn(P), \xinn(P))\in \Bil(T, \N)$, it suffices to prove that $\xinn(P) T\yinn(P)=0$.
Since $\N$ is strongly closed, we have, for all $Q\in \N$,
\begin{equation}\label{eq3}
Q^\perp T\yinn(Q)=0.
\end{equation}
 Hence 
$P'^\perp T\yinn(P)=0$, for all $P'\in \N$ for which $\yinn(P')=\yinn(P)$.  It now follows from \eqref{eq2}, that 
$\xinn(P) T\yinn(P)=0$, since $\N^\perp$ is also strongly closed.

Let $\S:=\{(P',Q)\in \Bil(T, \N)\colon P^\perp\leq Q\}$. Observe that, by \eqref{eq3}, $(\yinn(P), P^\perp)\in \S$. Hence, 
to show that \eqref{e2} holds, it suffices to prove that $(P',Q)\preceq (\yinn(P), P^\perp)$, for all $(P',Q)\in \S$.
 Suppose then that  $(P', Q)\in \S$. It follows, by Lemma \ref{rem1} and \eqref{eq1}, that $P'\leq \yinn (Q^\perp)\leq \yinn(P)$,  since $Q^\perp \leq P$. Hence $(P',Q)\preceq (\yinn(P), P^\perp)$, as required.

Let  $\X:=\{(\yinn(P'), P'^\perp)\in \Bil(T, \N)\colon \yinn(P')=\yinn(P)\}$. Hence
$$ \wedge \X=  \bigl(\yinn(P), \vee\{P'^\perp\in \N^\perp \colon \yinn(P')=\yinn(P)\}\bigr).
$$
It now follows from definition \eqref{eq2} that
$
\wedge \X=(\yinn(P), \xinn(P))$.
\end{proof}

Define the  \emph{kernel map $\omega_{T, {\tiny \N}}\colon \N\to  \Bil(T,\N)$ of  the operator $T$ relative to $\N$} by 
\begin{equation}\label{ed1}
\omega_{T, {\tiny \N}}(P)=(\varphi_{T, {\tiny \N}}(P), \psi_{T, {\tiny \N}}(P)).
\end{equation}

The \emph{kernel set $\Omega_{T, {\tiny \N}}$ of  the operator $T$ relative to $\N$} is the subset of $\Bil(T,\N)$ consisting of the image  of $\N$ under the map $\omega_{T, {\tiny \N}}$, that is,
 \begin{equation}\label{e001}
\Omega_{T, {\tiny \N}}=\{\omega_{T, {\tiny \N}}(P) \colon P\in \N\}.
\end{equation}
The set $\Omega_{T, {\tiny \N}}$ together with the restriction of the ordering of $\Bil(T,\N)$ is itself a partially ordered set.

 In the sequel, we write $\omega_{T}$ and $\Omega_{T}$, since it will be clear which subspace lattice we shall be considering. For simplicity, it will  be frequent to refer to $\omega_{T}$ and $\Omega_{T}$ as, respectively,  the kernel map and the kernel set. It is  also worth noticing that 
 \begin{equation}\label{eq_extreme_elements}
     \xinn (0)=I=\yinn(I).
 \end{equation}

\begin{lem}\label{reach_charact}
Let $\nest$ be a  subspace lattice,  let $T$ be an operator in $B(\H)$
 and let $\omega_{T}\colon \N\to  \Bil(T,\N)$ be the kernel map of $T$  relative to $\N$. The following assertions hold.
 
\begin{enumerate}[label=(\roman*)]
\item The kernel map $\omega_{T}$ is an order homomorphism such that, for all $P\in \N$, 
 $
\omega_T(\psi_T(P)^\perp)=\omega_T(P).
$
 \item For projections $P_1, P_2\in \nest$ with $P_1< P_2$, if $\omega_T(P_1)\neq \omega_T(P_2)$, then $\yinn(P_1)<\yinn(P_2)$ and $\xinn(P_2)<\xinn(P_1)$.
\item If $\nest$ is a nest, then, for all $P\in \N$, there exists $x\in \H$ such that  $\omega_T(P)=(P_x, P_{Tx}^\perp)$.
 \end{enumerate}
\end{lem}

\begin{proof}  (i) Let $P_1, P_2\in \N$ be such that $P_1\leq P_2$. Since, by Lemma \ref{rem1}, $\yinn(P_1)\leq \yinn (P_2)$ and $\xinn(P_2)\leq \xinn (P_1)$, we have immediately that $\omega_T(P_1)\preceq \omega_T(P_2)$, that is, $\omega_{T}$ is an order homomorphism.

Observe that, by  \eqref{extra1}, $\yinn(\xinn(P)^\perp) = \yinn (P)$. It now follows from  \eqref{eq2} that 
 $
\psi_T(\psi_T(P)^\perp)=\psi_T(P).
 $
Hence 
\begin{equation}\label{eprim}
\omega_T(\psi_T(P)^\perp)=\omega_T(P).
 \end{equation}

 (ii) Let $P_1, P_2\in \N$ be projections such that $P_1<P_2$ and $\omega_{T}(P_1)\neq \omega_{T}(P_2)$.  We shall show that 
 $\yinn(P_1)< \yinn(P_2)$ and $ \xinn(P_2)< \xinn(P_1)$.
 
 Assume, firstly, that $\yinn(P_1)\neq\yinn(P_2)$. Hence, the only possibility is $\yinn(P_1)< \yinn(P_2)$. If $\xinn(P_1)=\xinn(P_2)$, then, by (i) of this lemma,
 $$
 \yinn(P_1)=\yinn(\xinn(P_1)^\perp)=\yinn(\xinn(P_2)^\perp)= \yinn(P_2),
 $$
 which contraditcs the assumption. Hence  $\xinn(P_2)<\xinn(P_1)$.   
 
 If  we start by assuming that $\xinn(P_1)\neq\xinn(P_2)$, then it follows immediately from \eqref{eq2} that $\yinn(P_1)<\yinn(P_2)$.

(iii) Let $P$ be a projection in $\N$. By Proposition \ref{p_prelim}~(i), there exists $x \in \hilbsp$  such that $P_x = \yinn(P)$.   By Proposition \ref{l_bil}, for all $z\in \H$ with $P_z\leq \yinn(P)$, 
  $$\xinn(P)T z=\xinn(P)TP_z z=\xinn(P)T\yinn(P)z=0
  $$
 and, consequently,   $\xinn(P) \leq P^\perp_{Tx}$. If $\xinn(P) = P^\perp_{Tx}$, then the result is proved. 
 
 Suppose that, on the other hand,  $\xinn(P) < P^\perp_{Tx}$. We have that either $\psi_T(P)<\psi_T(P)_+$ or $\psi_T(P)=\psi_T(P)_+$.\\
 Case 1. $\psi_T(P)<\psi_T(P)_+$ or, equivalently, $(\psi_T(P)^\perp)_-<\psi_T(P)^\perp$.\\
 By   \eqref{eq2}, there exists a non-zero $y\in \varphi_T(P)$ such that 
 $$
 \bigl(\psi_T(P)^\perp -(\psi_T(P)^\perp)_-\bigr)Ty\neq 0,
 $$
from which follows that $P_{Ty}=\psi_T(P)^\perp$. 
If $P_y=\varphi_T(P)$, then the proof is finished. If $P_y<\varphi_T(P)$, then   set $z=x+y$. Hence  $P_{z}=\varphi_T(P)$ and 
 $P_{Tz}=\psi_T(P)^\perp$. \\
 Case 2. $\psi_T(P)=\psi_T(P)_+$ or, equivalently, $(\psi_T(P)^\perp)_-=\psi_T(P)^\perp$.\\
Let $(Q_n)$ be a strictly increasing sequence in $\N$ such that  $\xinn(P)^\perp=\vee\{Q_n\colon n\in \mathbb{N}\}$ (cf.   \cite[Proposition 3 and Corollary 4]{BCW}). 
 
Let 
 $$
 \Gamma=\bigl\{k\in \nats\colon (Q_{k+1} - Q_{k})T\yinn(P)\neq \{0\}\bigr\}.
 $$
 Since $\xinn(P) < P^\perp_{Tx}$, it follows from \eqref{eq2} that the set $\Gamma$ is infinite. For each $k\in \Gamma$, choose  
 $x_k\in \yinn(P)$ such that 
$$
(Q_{k+1} - Q_{k})Tx_{k}\neq 0,
$$
  $\|x_{k}\|=1$, and 
 let  $y=\sum_{k\in \Gamma}\frac{1}{k^2} x_k$. 
   
  A reasoning similar to that in the proof of Proposition \ref{p_prelim} (i) yields $\xinn(P)=P_{Ty}^\perp$. If $P_y=\yinn(P)$, then $(\yinn(P), \xinn(P))=(P_y, P_{Ty}^\perp)$, as required.
  
  If $P_y<\yinn(P)$, then take $z=x+y$. Obviously, $P_z=P_x=\yinn(P)$ and $P_{Tz}=P_{Ty}=\xinn(P)^\perp$, concluding the proof. 
  \end{proof}

\begin{rem}\label{rem8}
 Observe  that, if $\nest$ is totally ordered, that is, if $\nest$ is a nest, then $\Omega_T$ must be totally ordered, since, as seen in (i) of the above lemma, the map $\omega_T$ is an order homomorphism. 
   \end{rem}
   
\begin{thm}
\label{finite_contour}
Let $T$ be a  rank-$n$ operator in $\boundop$ and let $\nest$ be a nest. Then, the kernel set $\Omega_T$ of the operator $T$ relative to $\N$ has cardinality less than or equal to $n+1$. 
\end{thm}
\begin{proof} If $n=0$, then $T=0$, $\Omega_0=\{(I,I)\}$ and the assertion holds. 

Let now $n$ be positive. 
Firstly notice that, by Lemma \ref{reach_charact}~(i), for all $P_1,P_2\in \N$ with $\omega_T(P_1)\neq \omega_T(P_2)$, we have 
$\yinn(P_1) \neq \yinn(P_2)$ and $\xinn(P_1) \neq \xinn(P_2)$. Recall also that, by Lemma  \ref{reach_charact}~(iii), for each $P\in \N$, there exists $x \in \hilbsp$ such that $P_x = \yinn(P)$ and $P_{Tx}^\perp = \xinn(P)$.

Suppose now that the set $\llinner{T}$ has cardinality $m > n+1$.  By the above considerations, it is possible to find a subset 
$S=\{Tx_i\colon i=1, \dots, m\}$ of  $\image T$ having cardinality $m$ and such that $P_{Tx_i} \neq P_{Tx_j}$, for $i,j=1,\dots n$ with $i\neq j$. Hence there exist, at least,  $m-1>n$ non-zero elements in $S$. By 
 Proposition \ref{p_prelim} (i), it follows  that these $m-1$ elements are linearly independent, which contradicts the hypothesis of $T$ having rank equal to $n$.
\end{proof}
 
\begin{rem} Theorem \ref{finite_contour} cannot be improved, as will be clear from the example below. Let $\nest$ consist of the projections $0<P_1<P_2<P_3<P_4=I$, where $P_i, 1\leq i\leq 4$, corresponds to the span of the vectors $e_1,\dots, e_i$ of the canonical basis of $\mathbb{C}^4$. 
Consider the operators 
$$
T_1=\begin{bmatrix}
0&1&0&1\\
0&0&1&1\\
0&0&1&1\\
0&0&1&1
\end{bmatrix},
\qquad\qquad
T_2=\begin{bmatrix}
0&1&0&1\\
0&0&1&1\\
0&0&1&2\\
0&0&1&1
\end{bmatrix}
$$
on $\mathbb{C}^4$ whose ranks are, respectively, 2 and 3. However, both have the same  kernel set
$$
\Omega_{T_1}=\{(P_1,I),(P_2,P_1^\perp),(I,0)\}=\Omega_{T_2}.
$$
\end{rem}
\qquad \qquad \\
For each  $P$ in a nest $\N$, define the projection  $\sigma_T(P)$ by
\begin{equation}\label{eqq}
 \sigma_T(P)=\vee\{P'\in \N\colon \yinn(P')=\yinn(P)\}.
 \end{equation}
 It follows  from Lemma \ref{rem1}  that the map $P\mapsto \sigma_T(P)$ is an order homomorphism on $\nest$. Consequently, by Lemma \ref{reach_charact}~(i), for all $P$ in $\nest$,  $\omega_T(P) \preceq \omega_T(\sigma(P))$.

\begin{thm}\label{sw_tight_n_decomp}
Let $T$ be a rank-$n$ operator in $\boundop$ and let $\nest$ be a nest. Then, there exist   operators $\tens{x_r}{y_r} $,  $r=1,...,m$, such that $T=\sum_{r=1}^{m}\tens{x_r}{y_r}$ 
 and $\omega_T(P)\in \Bil(\tens{x_r}{y_r}, \N)$, for  all $P\in \N$. 
Moreover, if the nest is continuous then, for all  $r =1,...,m$, there exists $(\Phi_r, \Psi_r)\in \Omega_T$ such that  $\tens{x_r}{y_r} = \sigma(\Psi_r^\perp)(\tens{x_r}{y_r} )\ortho{\Phi_r}$.
\end{thm}

\begin{proof} If   $n=0$,  the assertion is trivially true. Suppose then that $n$ is a positive integer and let $T=\sum_{i=1}^{n}\tens{z_i}{w_i}$. 

In this case, there exists $P\in \N$ such that $\xinn(P)<I$. In fact, if $\psi(P)=I$ for all $P\in \N$, then $\psi_T(I)=I$. Hence, by Proposition \ref{l_bil},
$$
0=\psi_T(I)T\varphi_T(I)=ITI=T,
$$
which contradicts the assumption that $T\neq0$.

 Let $k$ be the number of elements in the proper subset $\{\omega_T(P)\colon \xinn(P) \neq I\}$  of $\Omega_T$. Using \thmref{finite_contour} and \eqref{eq_extreme_elements}, there exists an integer $k$, with $1\leq k<n+1$, such that
 $$\{\omega_T(P)\colon \xinn(P) \neq I\}=\{(\Phi_j,\Psi_j)\colon j=1,\dots, k\}.
 $$
 
 If $k=1$, then 
 $$\{\omega_T(P)\colon \xinn(P) \neq I\}=\{(\Phi_1,\Psi_1)\}=\{\omega_T(P_1)\},
 $$
  for some non-zero projection $P_1\in \N$. If $k\neq 1$, let   $P_1<\dots<P_k$ be the projections such that  $\omega_T(P_1)\preceq\dots\preceq\omega_T(P_k)$ are the strictly ordered elements of  
 $\{\omega_T(P)\colon \xinn(P) \neq I\}$ 
 (cf. Lemma \ref{reach_charact} (ii)).
  
 For $j=1,\dots, k$, let  $\omega_T(P_j)=(\Phi_j,\Psi_j)$, and let $\omega_T(0)=({\Phi}_{0}, I)$.  
We have 
\begin{align*}
T&=\sum_{j=1}^k T ({\Phi}_j-{\Phi}_{j-1})\\
&=\sum_{j=1}^k ({\Psi}_j+{\Psi}_j^\perp) T ({\Phi}_j-{\Phi}_{j-1})\\
&=\sum_{j=1}^k {\Psi}_j^\perp T ({\Phi}_j-{\Phi}_{j-1}),
\end{align*}
that is, 
\begin{equation}\label{e15}
T
=\sum_{j=1}^k \sum_{i=1}^{n}\tens{{\Psi}_j^\perp z_i}{ ({\Phi}_j-{\Phi}_{j-1})w_i}.
\end{equation}
 Hence we have written the operator $T$ as the sum of $m$ rank-1 operators with $m\leq k\times n$. 

We show next that, for each   
$$
\tens{{\Psi}_j^\perp z_i}{ ({\Phi}_j-{\Phi}_{j-1})w_i}
$$
 in \eqref{e15}, the pairs 
 $({\Phi}_{0}, I)$ and  $({\Phi}_{l}, {\Psi}_{l})$,  where $l=1,\dots, k$, lie in $\Bil\bigl(\tens{{\Psi}_j^\perp z_i}{ ({\Phi}_j-{\Phi}_{j-1})w_i}, \N\bigr)$.
 
We begin by proving that  $({\Phi}_{0}, I)\in \Bil\bigl(\tens{{\Psi}_j^\perp z_i}{ ({\Phi}_j-{\Phi}_{j-1})w_i}, \N\bigr)$. 
We have  
$$
I\tens{{\Psi}_j^\perp z_i}{ ({\Phi}_j-{\Phi}_{j-1})w_i}{\Phi}_{0}=\tens{{\Psi}_j^\perp z_i}{ ({\Phi}_{0}-{\Phi}_{0})w_i}=0,
$$
since it is always the case that  ${\Phi}_{0}\leq {\Phi}_{j-1}\leq {\Phi}_j$.

 Let $\tens{{\Psi}_j^\perp z_i}{ ({\Phi}_j-{\Phi}_{j-1})w_i}$ be an  operator in \eqref{e15} and let $l$ be an integer such that $1\leq l\leq k$.   We have two possibilites: $j\leq l$ or $l<j$. 
 
Suppose firstly that $j\leq l$. Then 
\begin{equation}\label{e16}
{\Psi}_{l}\tens{{\Psi}_j^\perp z_i}{ ({\Phi}_j-{\Phi}_{j-1})w_i}{\Phi}_{l}=0,
\end{equation}
since ${\Psi}_{l}{\Psi}_j^\perp=0$.
Now let $l<j$, that is, $l\leq j-1$. In this case, 
$$
{\Phi}_{l}({\Phi}_j-{\Phi}_{j-1})w_i=0,
$$
since ${\Phi}_{l}\leq {\Phi}_{j-1}$. Consequently, \eqref{e16} holds also when $l<j$.

Suppose now that the nest $\N$ is continuous. We claim that for each $j=1,\dots, k$, 
$\Psi_j^\perp
=\sigma(\Psi_{j-1}^\perp)$.

Observe that, by Lemma \ref{reach_charact} (i), (ii), 
$$
{\yinn}(({\Psi_{j-1}}){^\perp})={\Phi}_{j-1}< {\Phi}_j=\yinn({\Psi_j}^\perp).
$$
Hence, by the definitions \eqref{eq2}, \eqref{eqq} , we have 
$$
\Psi_{j-1}^\perp\leq \sigma(\Psi_{j-1}^\perp)\leq \Psi_j^\perp.
$$
Consequently, by Lemma \ref{reach_charact} (i), either 
$$\varphi_T\bigl(\sigma(\Psi_{j-1}^\perp)\bigr)=\varphi_T(\Psi_{j-1}^\perp)
$$ or

$$\varphi_T\bigl(\sigma(\Psi_{j-1}^\perp)\bigr)=\varphi_T(\Psi_j^\perp).
$$

 If $ \sigma(\Psi_{j-1}^\perp)<\Psi_j^\perp$, then there would exist a projection $Q$ such that 
$$
   \sigma(\Psi_{j-1}^\perp)<Q<\Psi_j^\perp,
$$
since the nest $\N$ and, consequently, the nest $\N^\perp$ are continuous.

It follows that either 
$$\yinn(Q)=\yinn(\Psi_j^\perp)
$$ or 
$$\yinn(Q)=\yinn\bigl(\sigma(\Psi_{j-1}^\perp)\bigr).
$$ 
Recall here that $\Omega_T$ consists of $k+1$ elements satisfying 
$$
({\Phi}_{0}, I)\preceq \dots\preceq ({\Phi}_{j-1}, {\Psi}_{j-1})\preceq ({\Phi}_j, {\Psi}_j)\preceq\dots\preceq (I, \Phi_k)
$$ 
and that, by Lemma \ref{reach_charact} (ii), this order is strict.

In the first case $Q$ would have to coincide with $\Psi_j^\perp$, and in the second $Q$ would have to be equal to $\sigma(\Psi_{j-1}^\perp)$, contradicting  the initial assumption of $Q$ being different from these projections.

Finally, 
$$
\tens{{\Psi}_j^\perp z_i}{ ({\Phi}_j-{\Phi}_{j-1})w_i}=
\sigma(\Psi_{j-1}^\perp)\bigl(\tens{{\Psi}_j^\perp z_i}{ ({\Phi}_j-{\Phi}_{j-1})w_i})\Phi_{j-1}^\perp,
$$
as required.
\end{proof}

\begin{rem}\label{remex}It is worth to make a note of the following fact already outlined in the proof above. Let $\nest$ be a continuous nest,  let $T\neq 0$ be a rank-$n$ operator and let $\Omega_T=\{(\Phi_j, \Psi_j)\colon j=1,\dots, k\}$. Then, for each $j=1,\dots, k$, we have
$\Psi_j^\perp
=\sigma(\Psi_{j-1}^\perp)$.
\end{rem}

\begin{lem}\label{l_px2}
Let $\N$ be a continuous nest, let $T$ be a finite rank operator with kernel set $\Omega_T=\{(\Phi_j, \Psi_j)\colon j=1, \dots, k\}$ and let   
 $P$ be a projection in $\nest$ such that 
  $\Phi_{j}<P\leq \Phi_{j+1}$ for some $j$. Then there exists $x\in \H$ such that 
$
\hat{P}_{x}=\Phi_{j}$,  $P_{x}=P$ and $P_{Tx}=\sigma(\Psi_{j}^\perp)
$.
\end{lem}

\begin{proof}
 Let $1\leq j< k$ be an integer and let $P\in \nest$ be a projection such that $\Phi_{j}<P\leq \Phi_{j+1}$. By Proposition \ref{p_prelim}, there exists $z\in \H$ such that  
$
\hat{P}_{z}=\Phi_{j}$ and  $P_{z}=P$. 

The proof is complete if $P_{Tz}=\sigma(\Psi^\perp_{j})$.  If this is not the case, then 
$$P_{Tz}<\sigma(\Psi^\perp_{j})=\Psi^\perp_{j+1}$$ (cf. Remark \ref{remex}). 
Let $(Q_n)$ be a strictly increasing sequence in  $\nest$ such that $\sigma(\Psi^\perp_{j})=\vee\{Q_n\colon n\in \mathbb{N}\}$ (cf. \cite[Proposition 3 and Corollary 4]{BCW}).

Define 
$$
\Gamma=\{n\in \mathbb{N}\colon (Q_{n+1} -  Q_n)T(P-\Phi_{j})\neq 0\}.
$$
This set is infinite. In fact, if one assumed that $\Gamma$ is finite, then there would exist a projection $P'\in \nest$ such that 
\begin{equation}\label{extra}
\Psi^\perp_{j} <P'<\sigma(\Psi^\perp_{j})
\end{equation}
with 
$
P'^\perp T(P-\Phi_{j})=0
$. But, by \eqref{eqq}, this would imply that $\sigma(\Psi^\perp_{j})<P'$, which  contradicts \eqref{extra}.

Let $(x_n)$ be a sequence with $x_n\in (P-\Phi_{j})(\H)$ such that, for all $n\in \mathbb{N}$,  $\|x_n\|=1$ and
$$
(Q_{n+1} - Q_n)T(P-\Phi_{j})x_n\neq 0.
$$
Let $y=\sum_{n=1}^\infty \frac{1}{n^2} x_n$.
It follows that  $P_{Ty}=\sigma(\Psi^\perp_{j})$. 

Setting $x=(P - P_{y} + \hat{P}_{y} - \Phi_{j})z+y$  ends the proof.
\end{proof}

\section{Operator decomposition}\label{s_decomp}
Recall that, given a nest $\N$ and the corresponding nest algebra $\nestalg$,  a subspace $\M$ of $B(\H)$ is called a Lie $\nestalg$-module if $[\M, \nestalg]\subseteq \M$. Here $[\,,]$ denotes the Lie bracket, that is, for $A, B\in B(\H)$, 
$$
[A,B]=AB-BA.
$$

 A finite rank operator $T$ in a subset $\X$ of $B(\H)$ 
  is said to be  \emph{decomposable in $\X$} if $T$ can be written as the sum of finitely many rank-1 operators in $\X$. 
   The set $\X$ is  \emph{decomposable} if every finite rank operator in $\X$ is decomposable.

We apply here the results and techniques of Section \ref{s_kernelmap} to prove the main theorem of this section:

\begin{thm}
\label{rankndecomp}
Let $\nest$ be a continuous nest. Then every norm closed Lie $\nestalg$-module is decomposable.
 \end{thm}

In other words,  every finite rank operator lying in a norm closed Lie module of a continuous nest algebra  is a sum of finitely many rank-1 operators in the module.  In general,  the continuity of the nest cannot be avoided, as Theorem \ref{rankndecomp} can fail in a more general setting.  For example, the Lie ideal $\mathbb{C} I$ in the  algebra   of the $n\times n$ upper triangular complex matrices is not decomposable (see also Remark \ref{rem003} below).

Before being able to prove Theorem \ref{rankndecomp}, we need to consider firstly some auxilary results.

\begin{lem}
\label{allrank1}
Let $\nest$ be a continuous nest  and let $\module$ be a norm closed Lie $\nestalg$-module. If $\tens{x}{y}$ is a rank-1 operator in $\module$, then, for all $z, w \in \hilbsp$,   $P \in \nest$, the operator $\ortho{P}(\tens{P_x z}{\ortho{\jproj{P}_y} w})P$ lies in $\module$.
\end{lem}
\begin{proof}
Firstly, notice that, if $z=0$, $w=0$, $\ortho{P}(\tens{x}{y})P = 0$, $P_x \leq \jproj{P}_y$,  or $P=0,I$, the assertion is trivially verified. Thus, assume otherwise. 

Observe that we can also assume that $z = \ortho{P}z = P_x z$ and $w = Pw = \ortho{\jproj{P}_y} w$. It follows that 
\begin{equation}\label{eee}
\hat{P}_y\leq \hat{P}_w\leq P_w\leq P<P_z\leq P_x.
\end{equation}
By \cite[Lemma 1]{OS},  it is also the case that, for all $P\in \N$, $\ortho{P}(\tens{x}{y})P \in \module$.

We begin by showing that $\ortho{P}(\tens{P_x z}{y})P \in \module$.

Case 1. $P_z < P_x$.

   For $a\in\H$, consider the operator $\tens{z}{\ortho{P_z}a}\in \nestalg$.
  We have that $\lieprod{\tens{z}{\ortho{P_z}a}}{\ortho{P}(\tens{x}{y})P}$ lies in $\M$ and
  
\begin{align*}
\lieprod{\tens{z}{\ortho{P_z}a}}{\ortho{P}(\tens{x}{y})P} &= \inner{\ortho{P}x}{\ortho{P_z}a} \; \tens{z}{Py} - \inner{z}{Py} \; \tens{\ortho{P}x}{\ortho{P_z}a} \\
&= \inner{\ortho{P}x}{\ortho{P_z}a} \; \tens{z}{Py},
\end{align*}
since $\inner{z}{Py}= \inner{\ortho{P}z}{Py}= 0$.

Recalling  that $P < P_z < P_x$, we can choose $a\in \H$ such that $a=\ortho{P_z}x\neq 0$. Hence  $\inner{\ortho{P}x}{\ortho{P_z}a} = \norm{\ortho{P_z}x}^2 \neq 0$ and, therefore,  $\tens{z}{Py}=\ortho{P}(\tens{P_x z}{y})P$ lies in  $\module$. 

 Case 2. $P_z = P_x$.
 
  Following the proof of  \cite[Theorem 3.1]{O2},  observe that
$$
P_x =P_z =\vee\{Q \in \nest\colon  Q < P_z\},
$$
from which follows that there exists a sequence $(z_n)$ in 
$$\union_{Q\in \tiny{\N}, Q < P_z} Q(H)$$
 such that $z_n= \ortho{P}z_n = P_x z_n$ which converges to $z$ in the norm topology. 

Hence, by Case 1., $(\tens{z_n}{Py})$ is a convergent sequence in the norm closed Lie $\nestalg$-module $\module$, from which follows that its limit $\tens{z}{Py}$ also lies in $\module$.

We have shown that $\ortho{P}(\tens{P_x z}{y})P \in \module$ and will show next that $\ortho{P}(\tens{P_x z}{\ortho{\jproj{P}_y} w})P \in \module$.

As observed before, we can assume that $w=\hat{P}_y^\perp w$. Hence either $\hat{P}_y<\hat{P}_w$ or $\hat{P}_y=\hat{P}_w$. 

Suppose now that $\jproj{P}_y < \jproj{P}_w$. For $b \in \hilbsp$, consider the operator $\tens{\jproj{P}_w b}{w}$ in $\nestalg$. Then 
$\lieprod{\tens{z}{Py}}{\tens{\jproj{P}_w b}{w}}$ lies in $\M$  and 
\begin{align*}
\lieprod{\tens{z}{Py}}{\tens{\jproj{P}_w b}{w}}
 &= \inner{\jproj{P}_w b}{Py} \; \tens{z}{w} - \inner{z}{w} \; \tens{\jproj{P}_w b}{Py}\\
&= \inner{\jproj{P}_w b}{Py} \; \tens{z}{w},
\end{align*}
since $\inner{z}{w} = \inner{\ortho{P}z}{Pw}= 0$.

 Observing that $\jproj{P}_y < \jproj{P}_w < P$ (cf. \eqref{eee}), we can choose $b = \jproj{P}_wy\neq 0$. 
Hence  
 $\inner{\jproj{P}_w b}{Py} = \norm{\jproj{P}_w y}^2 \neq 0$ and 
 it follows that  $\tens{z}{w} \in \module$. 
 
If on the other hand $\jproj{P}_y = \jproj{P}_w$, then 
$$
\hat{P}_w^\perp = \hat{P}_y^\perp = \vee\{ \ortho{Q} \colon Q \in \nest, \jproj{P}_y < Q\},
$$
 from which follows that there exists a sequence $(w_n)$ in 
 $$\union_{Q\in {\tiny \N}, \jproj{P}_y < Q} \ortho{Q}(H)$$
  converging to $w$ in the norm topology and  such that $w_n = Pw_n = \jproj{P}_y^\perp w_n$. 
 Hence $(\tens{z}{w_n})$ is a convergent sequence in $\module$ whose limit also lies in $\module$.
\end{proof}

\begin{lem}
\label{allrank1_nalg}
Let $\nest$ be a continuous nest  and let $\module$ be a norm closed Lie $\nestalg$-module. If $\tens{x}{y}$ is a rank-1 operator in $\module$, then, for all $z, w \in \hilbsp$ and all  $P \in \nest$, the operator $P(\tens{P_x z}{\ortho{\jproj{P}_y} w})\ortho{P}$ lies in $\module$.
\end{lem}

\begin{proof} 
Notice that, as in the proof of Lemma \ref{allrank1}, if $P=0,I$, then  the assertion is trivially verified. Thus, assume that $P\neq 0,I$. 

If $\tens{x}{y} \in \nestalg$ or, equivalently, if $P_x\leq \jproj{P}_y$, the assertion is a consequence of   \cite[Theorem 3.1]{O2}.

Assume then that $\jproj{P}_y < P_x$. It suffices to show that
  there exists a rank one operator $\tens{e}{f}$ such that $\tens{e}{f}=P(\tens{e}{f})\ortho{P} \in \module$, with $P_e = P=\jproj{P}_f$, and apply \cite[Theorem 3.1]{O2}. Notice that, in this case, $\tens{e}{f}$ lies in the  norm closed ideal $\M\cap\nestalg$, since $P(\tens{e}{f})\ortho{P} \in \nestalg$(cf. \cite[Lemma 3.3]{ringrose1965}).
  
  Let    $a,b \in \hilbsp$ and consider the Lie bracket

\begin{align}\label{a1}
\lieprod{\tens{a}{b}}{\tens{x}{y}} &= \inner{x}{b}\tens{a}{y} - \inner{a}{y}\tens{x}{b}.
\end{align}

We analyse firstly the  cases  $P = \jproj{P}_y$ (Case 1), $P = P_x$ (Case 2) and $\jproj{P}_y<P<P_x$ (Case 3). The remaining possibilities,  $P<\jproj{P}_y$ or $P_x<P$, will be dealt with in the last part of the proof.

Case 1. $P = \jproj{P}_y$.
 
Let $a\in\H$ be such that $P_{a} = P$ (cf. Proposition \ref{p_prelim} (i)).
 Hence $\inner{a}{y} = \inner{\jproj{P}_ya}{\ortho{\jproj{P}_y}y} = 0$.  
 
 Let $b=\ortho{P}x$. In these circumstances, 
 $$
 \inner{x}{b}=\inner{x}{\ortho{P}x} = \norm{\ortho{P}x}^2 \neq 0,
 $$
  since $P=\hat{P}_y<P_x$. 
  Moreover, since $a\otimes b\in \nestalg$, it follows from 
   \eqref{a1} that  $\tens{a}{y} \in \module$. Setting $e=a$ and $f=y$ concludes the proof of this case.
 
Case 2. $P = P_x$.

 Let $a= Py$ and let $b \in \H$ be  such that $\jproj{P_{b}} = P$ (cf. Proposition \ref{p_prelim} (i)). 
 Hence, $a\otimes b\in \T(\N)$, $\inner{a}{y} = \norm{Py}^2 \neq 0$ and $\inner{x}{b} = \inner{Px}{\ortho{P}b} = 0$. It follows from \eqref{a1} that  $\tens{x}{b} \in \module$. The proof is concluded by setting $e=x$ and $f=b$.

 Case 3. $\jproj{P}_y<P<P_x$. 
  
 We show now that $a$ in \eqref{a1} can be chosen with $P_{a}=P$ and $\inner{a}{y}=0$.  By \cite[Proposition 3 and Corollary 4]{BCW}, there exists  a strictly increasing sequence $(Q_n)$ in $\nest$ such that $P=\vee Q_n$. 
  Let $(a_n)$ be a sequence in $\H$ such that, for all $n\in \nats$, $a_n\in Q_{n+1}\ominus Q_n$, $\|a_n\|=1$ and $\inner{a_n}{y}=0$. Observe that, if   it were the case that for some $n$ there did not exist $a_n$ orthogonal to $y$, then $\dim(Q_{n+1}\ominus Q_n) = 1$ which is an impossibility, given the continuity of the nest.
  
  Then $a=\sum_{n=1}^\infty \frac{1}{n^2} a_n$ is orthogonal to $y$ and $P_a=P$. It is immediate that $a$ satisfies the requirements.
  
  Let $b =\ortho{P}x$. We have  $\inner{x}{b}=\inner{x}{\ortho{P}x} = \norm{\ortho{P}x}^2 \neq 0$. Hence, $a\otimes b\in \T(\N)$ and it follows from \eqref{a1} that $\tens{a}{y} \in \module$. The proof is concluded by setting $e=a$ and $f=y$.

 Finally, consider the remaining possibilities  $P<\jproj{P}_y$ or $P_x<P$.
 
 If $P<\jproj{P}_y$, then 
 $$
 P(\tens{P_x z}{\ortho{\jproj{P}_y} w})\ortho{P}
 =P_y(\tens{P_x Pz}{\ortho{\jproj{P}_y} w})\ortho{\jproj{P}_y}
 $$
 lies in $\M$ by Case 1. of this proof.

 If $P_x<P$, then 
 $$
 P(\tens{P_x z}{\ortho{\jproj{P}_y} w})\ortho{P}
 =P_x(\tens{P_x z}{P^\perp\ortho{\jproj{P}_y} w})\ortho{P_x}
 $$
 lies in $\M$ by Case 2. above.
\end{proof}

Before proving the next  theorem, we need a definition. 
Let $\mathcal{F} = \{(P_i)_{i \in \{0,...,k\}}:k \in \nats, P_i \in \nest, 0=P_0<...<P_{k}=I\}$ be the set  of totally ordered finite subsets of  a nest $\nest$ beginning with the 0 projection and ending with $I$, seen as a net when ordered by inclusion. Given an operator $T$ in $B(\H)$ and $F=(P_i)_{i \in \{0,...,k\}}$ in $\F$, define
\begin{align}
U_F (T) &= \sum_{i = 1}^k P_{i-1}T(P_i-P_{i-1})\label{a2}\\
L_F (T) &= \sum_{i = 1}^k \ortho{P_{i}}T(P_i-P_{i-1})\label{a3}\\
D_F (T) &= \sum_{i = 1}^k (P_i-P_{i-1})T(P_i-P_{i-1})= T - (U_F (T) + L_F (T)) \label{a4}
\end{align}
(cf. \cite{davidsonbook}). Observe that $T=U_F (T)+L_F (T)+D_F (T)$ and that $U_F (T)$ lies in $\nestalg$ whilst $L_F (T)$ lies in $\nestalg^*$.

\begin{thm}
\label{t_allrank1}
Let $\nest$ be a continuous nest, let $\module$ be a norm closed Lie $\nestalg$-module and let $\tens{x}{y}$ be a rank-1 operator in $\module$.
  Then, $\module$ contains all  rank-1 operators in $P_x\boundop \ortho{\jproj{P}_y}$.
\end{thm}

\begin{proof}
 Let  $\tens{a}{b}$ be a rank-1 operator  such that $\tens{a}{b} = P_x\tens{a}{b}\ortho{\jproj{P}_y}$. By \cite[Proposition 4.3]{davidsonbook}, the net $(D_F(\tens{a}{b}))$ converges to zero in the norm topology. It follows from \eqref{a4} that $(U_F(\tens{a}{b}) + L_F(\tens{a}{b}))$ converges to $\tens{a}{b}$ in the norm topology.

 By Lemmas \ref{allrank1}, \ref{allrank1_nalg}, for each $F$, the operators $U_F(\tens{a}{b}), L_F (\tens{a}{b})$ lie in the Lie $\nestalg$-module $\module$. Hence $\tens{a}{b} \in \module$, since $\module$ is  norm closed. 
\end{proof}

\begin{rem} \label{rem003}In general,  the continuity of the nest cannot be lifted. For example, Theorem \ref{t_allrank1} does not hold if $\nestalg$ is the algebra of the $6\times 6$  upper triangular complex matrices. In these circumstances,  the 
  smallest Lie $\nestalg$-module $\M$ containing the matrix unit $E_{65}$  consists of the zero trace matrices having the first four columns equal to zero. The module $\M$ does not contain  $E_{66}$, for example. 
\end{rem}

\begin{thm}
\label{rank1_module}
Let $\nest$ be a continuous nest, let $\module$ be a norm closed Lie $\nestalg$-module, let $T$ be a finite rank operator in $\module$ and let $(\Phi, \Psi)$ be a pair of projections  in  the kernel set $\Omega_T$. Then $\module$ contains all  rank-1 operators in  $\sigma(\Psi^\perp) \boundop \ortho{\Phi}$.
\end{thm}

\begin{proof}  Let $(\Phi, \Psi)$ be a pair in the kernel set  $\Omega_T$ and suppose that $\Phi<I$, since the theorem holds trivially when $\Phi=I$.   The proof is split into the cases $\Phi < \sigma(\Psi^\perp)$ (Case 1) and  $\sigma(\Psi^\perp)\leq \Phi$ (Case 2) below. 

 Case 1. $\Phi < \sigma(\Psi^\perp)$.\\
 Let $P', Q$ be any projections in $\nest$ such that $\Phi < Q < P' < \sigma(\Psi^\perp)$. By \cite[Lemma 1]{OS},  $\ortho{P'}TP' \in \module$. 

By Proposition \ref{p_prelim}, we can also choose  $h \in \hilbsp$  such that $\jproj{P}_h = \Phi$. Let $a\in \H$ and consider the operator $\tens{Q a}{\ortho{Q} P' h}$ which, by 
by \cite{ringrose1965}, Lemma 3.3, lies in $\nestalg$.
Notice that both $a$ and  $h$ can be chosen such that the operator $\tens{Q a}{\ortho{Q} P' h}$ is non-zero.
It follows that
\begin{align*}
\lieprod{\ortho{P'}TP'}{\tens{Q a}{\ortho{Q}P'h}} &= \ortho{P'}TP'(\tens{Q a}{\ortho{Q}P'h}) - (\tens{Q a}{\ortho{Q}P'h})\ortho{P'}TP' \\
&= \tens{\ortho{P'}TP'Q a}{\ortho{Q}P'h} = \tens{\ortho{P'}TQ a}{\ortho{Q}P'h} 
\end{align*}
lies in $\module$.

Observing that   $\Phi\neq I$ and that $\max\{\Phi'\colon (\Phi',\Psi')\in \Omega_T\}=I$,  let  $Q$  be such that 
$$
\Phi<Q\leq \min\{\Phi'\colon \Phi<\Phi'\}.
$$
By Lemma \ref{l_px2}, we can assume that  $a \in \hilbsp$  is such that $\jproj{P}_a = \Phi$, $P_a = Q$ and $P_{Ta} = \sigma(\Psi^\perp)$. It follows that $Qa=a$, 
 $P_{TQa} = \sigma(\Psi^\perp)$ and, since $P' < \sigma(\Psi^\perp)$, 
$$
P_{TQa} = P_{\ortho{P'}TQa} = \sigma(\Psi^\perp).
$$ 
By Theorem \ref{t_allrank1}, it follows that all rank-1 operators in 
\begin{equation}\label{3.5}
P_{\ortho{P'}TQa}B(\H)\hat{P}_{\ortho{Q}P'h}=\sigma(\Psi^\perp) B(\H)\hat{P}_{\ortho{Q}P'h}
\end{equation}
lie in $\M$.

Since 
$$
\ortho{\Phi} = \vee\{N^\perp \colon N\in \N,  \Phi < N\},
$$
 there exists a sequence $(h_n)$ in 
$$
\bigcup_{N \in \nest, \Phi < N}\ortho{N}(\hilbsp)$$
  converging to $h$ in the norm topology.  Moreover, the sequence can be chosen such that, for each $n\in \nats$, there exists $Q_n \in \nest$ with $\Phi<Q_n$ and 
  $h_n = \ortho{Q_n} h_n = P' h_n$.
  
  Hence 
  $$
  \tens{\ortho{P'}TQ a}{h_n}=\tens{\ortho{P'}TQ a}{\ortho{Q_n}P'h_n}
  $$ 
lies in the set  \eqref{3.5}. 
Hence  $(\tens{\ortho{P'}TQ a}{\ortho{Q_n}P'h_n})$   is a convergent sequence in $\module$ whose limit $\tens{\ortho{P'}TQ a}{h}$ also lies in $\module$. Now an immediate application of  Theorem \ref{t_allrank1} yields the result.
 
Case 2. $\sigma(\Psi^\perp)\leq \Phi$. \\
Recall that $\Phi <I$ and let 
 $P' \in \nest$ be such that $\Phi < P'$. 
  Let $z \in \hilbsp$ and let $w \in \hilbsp$ be such that $\hat{P}_w=P'$ (cf. Proposition \ref{p_prelim}).
  Since, by  \cite[ Lemma 1]{OS}, 
$\sigma(\ortho{\Psi}) T \bigl(\sigma(\ortho{\Psi}) \bigr)^\perp\in \module$, and by \cite{ringrose1965}, Lemma 3.3, $\tens{P'z}{\ortho{P'}w} \in \nestalg$, it follows
\begin{align*}
\lieprod{\sigma(\ortho{\Psi}) T \bigl(\sigma(\ortho{\Psi}) \bigr)^\perp}{\tens{P'z}{\ortho{P'}w}} 
&= \sigma(\ortho{\Psi}) T \bigl(\sigma(\ortho{\Psi}) \bigr)^\perp(\tens{P'z}{\ortho{P'}w})\\
&\qquad   - (\tens{P'z}{\ortho{P'}w})\sigma(\ortho{\Psi}) T \bigl(\sigma(\ortho{\Psi}) \bigr)^\perp \\
&= \sigma(\ortho{\Psi}) T \bigl(\sigma(\ortho{\Psi}) \bigr)^\perp (\tens{P'z}{\ortho{P'}w}) 
\end{align*}
lies in $\module$.

 By Lemma \ref{l_px2},  $z \in \hilbsp$ can be chosen such that $\jproj{P}_z = \Phi$, $P_z = P'$ and $P_{Tz} = \sigma(\Psi^\perp)$. 
 Hence,   
$$P_{T\sigma(\Psi^\perp)^\perp P'z} = P_{T\Phi^\perp P'z} = \sigma(\Psi^\perp),
$$  
since $  \sigma(\Psi^\perp) \leq \Phi<P'$.

 An application Theorem \ref{t_allrank1} to 
$\sigma(\Psi^\perp)  T\bigl(\sigma(\ortho{\Psi}) \bigr)^\perp (\tens{P'z}{\ortho{P'}w})$ yields that $\M$ contains all rank-1 operators in $\sigma(\Psi^\perp)B(\H) \ortho{P'}$. Now, a density argument similar to that of Case 1., concludes the proof. 
\end{proof}

We are now ready to prove  Theorem \ref{rankndecomp}.

\begin{proof}  
Let $\module$ be a norm closed Lie $\nestalg$-module and let $T$ be a rank-n operator in $\module$.  
Applying Theorem \ref{sw_tight_n_decomp}, we get a decomposition $T = \sum_{r=1}^m \tens{x_r}{y_r}$ such that,

 for all  $r \in \{1,...,m \}$, there exists $(\Phi_r, \Psi_r)\in \Omega_T$ with 
 $\tens{x_r}{y_r} = \sigma(\Psi_r^\perp)(\tens{x_r}{y_r}) \ortho{\Phi_r}$. 
It finally follows  from Theorem \ref{rank1_module} that $\tens{x_r}{y_r} \in \module$.
\end{proof}


\begin{thebibliography}{9}




 \bibitem{BO}
  J. Bra\v{c}i\v{c} and L. Oliveira, \textit{A characterization of reflexive spaces of operators},  Czech. Math. J., 68 (143) (2018), 257--266.

 \bibitem{BCW}
  D. Buhagiar, E. Chetcuti and H. Weber, \textit{The order topology on the projection lattice of a Hilbert space},  Topol. Appl.,  159 (2012), 2280--2289.
  
  \bibitem{CL}
 C. Chen and F. Lu, \textit{Decomposability of finite rank operators in Lie ideals  of nest algebras},  Integr. Equ. Oper. Theory, 81 (2015), 427--434.

\bibitem{davidsonbook}
  K. R. Davidson, Nest algebras, Longman, 1988.
  \bibitem{FMS}
C. K. Fong, C. R. Miers and A. R. Sourour, \textit{Lie and Jordan ideals of operators on Hilbert
space}, Proc. Amer. Math. Soc. (4) 84 (1982), 516--520.

\bibitem{Deg} 
	H. Deguang, 
	\textit{On $\mathcal{A}$-submodules for reflexive operator algebras}, 
	Proc. Amer.  Math. Soc. 104 (4) (1988), 1067--1070.


\bibitem{erdosdensity}
 J. A. Erdos, \textit{Operators of finite rank in nest algebras},  J. London Math. Soc., 43 (1968), 391--397.
  \bibitem{Erdos}
  J. A. Erdos, \textit{Reflexivity for subspace maps and linear spaces of operators}, Proc. Lond.
Math. Soc., III Ser. 52 (1986), 582--600.

\bibitem{EP}
 J. A. Erdos and S. C. Power, \textit{Weakly closed ideals of nest algebras},  J. Operator Theory, 7 (1982), 219--235.
 
 
\bibitem{Had} 
	D. Hadwin, 
	\textit{A general view of reflexivity}, 
	Trans. Amer. Math. Soc. 344 (1994), 325--360.

 \bibitem{Halmos}
  P. R. Halmos, \textit{Reflexive lattices of subspaces},  J. London Math. Soc., (2) 4 (1971), 257--263.
  

  \bibitem{HM}
 A. Hopenwasser and R. Moore,  \textit{Finite rank operators in reflexive operator algebras},  J. London Math. Soc., (2) 27 (1983), 331--338.
 
  \bibitem{LS} A.\ I.\ Loginov and V.\ S.\ Shulman, \textit{Hereditary and intermediate reflexivity of $W^*$-algebras}, Izv. Akad.
Nauk SSSR Ser. Mat., 39 (1975) 1260-1273 (in Russian).
 \bibitem{Long}
  W. Longstaff, \textit{Operators of rank one in reflexive algebras},  Canad. J. Math., (1) 28 (1976) 19--23.

  \bibitem{LL}
 F. Lu and S. Lu, \textit{Decomposability of finite-rank operators in commutative subspace lattice algebras},   J. Math. Anal. Appl., 264 (2001), 408--422.

   \bibitem{O1}
  L. Oliveira, \textit{Weak*-closed Jordan ideals of nest algebras},  Math. Nachr., 248-249 (2003) 129--143.
  \bibitem{O2}
  L. Oliveira, \textit{Finite rank operators in Lie ideals of nest algebras},  Houston J. Math., (2) 37 (2011) 519--536.
  
\bibitem{OS}
  L. Oliveira and M. Santos, \textit{Weakly closed Lie modules of nest algebras},  Oper. Matrices, (1) 11 (2017), 23--35
\bibitem{ringrose1965}
  J. R. Ringrose, \textit{On some algebras of operators}, Proc. London Math. Soc., 1965.
  
\bibitem{ST} V.\ S.\ Shulman and L.\ Turowska, \textit{Operator synthesis. I. Synthetic sets, bilattices and tensor algebras}, J. Funct. Anal. 209 (2004), 293--331.  
\end{thebibliography}
\end{document}